\documentclass[11pt]{amsart}

\usepackage{graphicx}
\usepackage{amssymb}
\usepackage{epstopdf}
\usepackage{color}
\usepackage{comment}
\usepackage{hyperref}
\usepackage{pdfpages}

\newtheorem{theorem}{Theorem}[section]
\newtheorem{lemma}[theorem]{Lemma}

\theoremstyle{definition}
\newtheorem{definition}[theorem]{Definition}

\theoremstyle{remark}
\newtheorem{remark}[theorem]{Remark}

\theoremstyle{example}

\theoremstyle{property}
\newtheorem{property}[theorem]{Property}

\theoremstyle{question}
\newtheorem{question}[theorem]{Question}

\newcommand{\BE}{{\mathbb{E}}}

\newcommand{\BN}{{\mathbb{N}}}

\newcommand{\CA}{{\mathcal{A}}}

\newcommand{\bae}{\begin{equation}\begin{aligned}}
\newcommand{\eae}{\end{aligned}\end{equation}}

\newcommand{\pr}{\mathbb{P}}
\newcommand{\var}{\textbf{Var}}
\newcommand{\Z}{\mathbb{Z}}

\newcommand{\dd}{\partial}
\newcommand{\note}[1]{{\color{red}{ \bf{ [Note: #1]}}}}

\DeclareMathOperator*{\Cov}{\mathrm{Cov}}

\makeatletter
\@addtoreset{section}{part}
\makeatother

\begin{document}

\numberwithin{equation}{section} 

\title{Domino Tile Placing on Graphs}
\author{Jacob J. Kagan}
\begin{abstract}
Given a graph $G$ we consider sequentially placing dimers on it, namely choosing a maximal independent subset of edges, i.e. edges that do not share common vertices. We study the number of vertices that do not belong to any edge found in the maximal set. We prove a CLT result for this model in the case when the underlying graph is $\Z^d$. 
\end{abstract}

\maketitle

\section*{Introduction}

Dimer models arise naturally in several problems. Dimer coverings of graphs were extensively considered in the frame of perfect matchings and therefore have considerable interest in various areas. In particular, planar graph dimer coverings were extensively studied, see \cite{kenyon2009lectures} for an overview.

We consider a different dimer-grpah model, instead of considering coverings we consider dimer placements on a graph and we are interested in the fraction of vertices that such a process leaves uncovered. The general problem of placing dimers and monomers on a graph, is known to be notoriously hard, and we are not aware of any previous results in this direction \footnote{A problem in the spirit of our model has been considered in \cite{krapivsky2016kinetics}, there the authors consider a kinetic model of adsorption.}.

The motivation for our model comes from chemistry. It has been experimentally observed that in the reaction of release of halogens from cobalt chains a non-negligible fraction of halogens was not released during heating\footnote{ See \cite{tulchinsky2017reversible} for a detailed account of the chemical aspects. In a private discussion with the first author of \cite{tulchinsky2017reversible}, it seems that the reported number of 20\% was inaccurate and a better estimate is 14.5\%, which is in remarkable agreement with the prediction of the model we consider, especially having in mind its simplicity.}. This suggested a cooperative mechanism of release, which gave rise to the 1 dimensional model we consider in the following section. 

Unlike the more general setting which we consider in later sections the 1-dimensional case can be analysed using generating functions, and it gives some explicit predictions that are given in the work of Tulchinsky et al. \cite{tulchinsky2017reversible}. We begin by reproducing  this result for the sake of completeness. 




\section{1-dimensional model}\label{sec:lin}

Consider a row of $n$ sites and a process of placing ``dimers" that is tiles that cover two neighbouring cites with no tile overlaps allowed. We construct the process as follows: at each step we uniformly choose a pair of adjacent sites, if both sites are free we cover them with a tile, otherwise we do nothing. We repeat this process until no more tiles can be placed.

We encode the process by considering $n$ zeros in a row, covering a pair with a tile corresponds to replacing the corresponding zeros with a pair of ones. We consider the number of zeros that remain after there are no more admissible substitutions. This corresponds to calculating the number of holes (we sometimes refer to them as ``monomers") the tilling process leaves.

Consider the first step of the process. Let us denote the pair as $(k+1,k+2)$, thus $k$ may take values $k\in \{0,\ldots, n-2 \}$. This step can be represented by the following picture:

\begin{gather*}\label{eq:rep}
 \underbrace{\overbrace{0\ldots 0}^{k} \cdot 0\cdot 0 \cdot\overbrace{0\ldots 0}^{n-k-2}}_n \\
\\
 \Downarrow\\
\underbrace{\overbrace{0\ldots 0}^{k} \cdot 1\cdot 1 \cdot\overbrace{0\ldots 0}^{n-k-2}}_n
\end{gather*}

Let us introduce a random variables $X_n$ which denotes the number of zeros left at the end of the process in an interval of length $n$ and $Y_1$ which denotes the place of the tile picked at step $1$. The above suggests the following equation:

\begin{equation}
\BE[X_n|Y_1 = k] = \BE[X_k|Y_1 = k] + \BE [X_{n-k-2}|Y_1 = k]
\end{equation}

It is important to note that $X_k$ and $X_{n-k-2}$ are independent and are independent of $Y_1$. We will use this fact extensively in the following analysis.
\subsection{An equation for the expectation.}
An interval of $n$ sites contains $n-1$ gaps. Note that each gap corresponds to a possible choice of a dimer (the site on the left and the site from the right). Assuming the dimers are uniformly chosen we can take the expectation of the recurrence equation and obtain:

\begin{align*}
 \BE X_n &= \frac{1}{n-1}\sum_{k = 0}^{n-2}\BE(X_k) + \BE(X_{n-k-2}) =  \frac{2}{n-1}\sum_{k = 0}^{n-2}\BE(X_k)
\end{align*}

Let us denote $e_n =\BE X_n $ and rewrite the above equation
\begin{equation} \label{eq:E[X_n]}
e_n = \frac{2}{n-1}\sum_{k = 0}^{n-2}e_k
\end{equation}
This equation can be simplified, by getting rid of the sum on the right hand side. Note that:
\[ (n+1)\cdot e_{n+2} = 2\sum_{k=0}^{n}e_k \]
\[ n\cdot e_{n+1} = 2\sum_{k=0}^{n-1}e_k \]
taking the difference we obtain

\begin{equation}\label{eq:rec}
(n+1)\cdot e_{n+2} -n\cdot e_{n+1}= 2\cdot e_{n}
\end{equation}
with the initial conditions $e_0 = 0$ (and $e_1 = 1$)

\subsection{solving the recurrence}
Let us introduce the generating function

\[A(x) = \sum_{n=0}^\infty e_n\cdot x^n\]
we can obtain a differential equation for $A(x)$ from the difference equation (\ref{eq:rec}). We do this by first multiplying our equation by $x^n$

\[(n+1)e_{n+2}x^{n} -n e_{n}x^n = 2e_n x^n \]
and then manipulating to adjust the powers and the indexes which results in the equation 
\[\frac{1}{x}(n+2)e_{n+2}x^{n+1} -(n+1)e_{n+1}x^n = 2e_n x^n+\frac{1}{x^2}e_{n+2}x^{n+2} -\frac{1}{x}e_{n+1}x^{n+1} \]
Summing it over $n$ and noting 

\begin{align*}
&\frac{1}{x}\sum_{n=0}^{\infty}(n+2)e_{n+2}\cdot x^{n+1}  = \frac{1}{x}(A(x)' -e_1), & &\frac{1}{x^2}\sum_{n=0}^{\infty}e_{n+2}\cdot x^{n+2} = \frac{1}{x^2}(A(x) - e_1x -e_0)\\ 
& \sum_{n = 0}^{\infty} (n+1)\cdot e_{n+1}\cdot x^{n} = A(x)', & &\frac{1}{x}\sum_{n=0}^{\infty} e_{n+1}x^{n+1} = \frac{1}{x}(A(x) - e_0)
\end{align*}
transforms equation (\ref{eq:rec}) to an ODE for the generating function $A(x)$: 
\[ \frac{1}{x}(A(x)' -e_1) - A(x)' = 2A(x) +  \frac{1}{x^2}(A(x) - e_1x -e_0) - \frac{1}{x}(A(x) - e_0) \]
using $e_0 = 0$  the ODE for A(x) reads:
\begin{align*}
\begin{cases}
A(x)'(x - x^2) = A(x)(2x^2 +  1 - x) \\
A(0) = e_0 = 0
\end{cases}
\end{align*}
This can be solved by separating the variables
\[ (\log A(x))'= \frac{A(x)'}{A(x)}= \frac{-2x^2+x-1}{x^2-x} = (-2x + \log x - 2\log (x-1))'\]
using $A(0) = 0$ the solution is given by
\[ A(x) = \frac{x}{(x-1)^2} e^{-2x}\]

To extract the $e_n$ series recall the following Taylor expansions:
\begin{align*}
\frac{x}{(x-1)^2} = \sum_{n=0}^\infty n\cdot x^n & & e^{-2x} = \sum_{n=0}^\infty \frac{(-2x)^n}{n!}
\end{align*}

plugging the expansions in 

\begin{align*}
A(x) = \sum_{n=0}^\infty e_n\cdot x^n &= \sum_{i = 0}^{\infty} i\cdot x^i \cdot\sum_{j = 0}^\infty \frac{(-2x)^j}{j!} \\
&= \sum_{n=0}^\infty x^n\sum_{i+j = n}i\cdot\frac{(-2)^j}{j!} \\
&= \sum_{n=0}^\infty x^n\sum_{k=0}^n (n-k)\cdot\frac{(-2)^k}{k!}
\end{align*}

We obtain
\[\BE X_n = \sum_{k=0}^n (n-k)\cdot\frac{(-2)^k}{k!} = n\sum_{k=0}^n\frac{(-2)^k}{k!} -2\sum_{k=0}^{n-1} \frac{(-2)^k}{k!} = \frac{n}{e^2}-\frac{2}{e^2} + o(1/n)\]
This implies a non vanishing concentration of monomers 
\[\lim_{n\to\infty}\frac{\BE X_n}{n} =  \frac{1}{e^2} \]
substituting the numerical value ($ e^2 \approx 7.389$ ) gives a concentration of 13.5\%. 

\subsection{Variance} 
Let us write down the equation for the variance of $X_n$. By definition:
\begin{align} \label{eq:var_def}
\var X_n = \BE X^2_n -(\BE X_n)^2 = \BE [ \BE [X^2_n|Y_1] ] -(\BE X_n)^2
\end{align}
using the recurrence relation we have 
\begin{align*}
\BE [ \BE [X^2_n|Y_1] ]  &= \frac{1}{n-1}\sum_{k = 0}^{n-2} \BE(\BE[X_k|Y_1] + \BE [X_{n-2-k}|Y_1])^2 \\
&=\frac{1}{n-1}\sum_{k = 0}^{n-2} \BE[X^2_k] + \BE[X^2_{n-2-k}] + 2\BE[X_k]\BE[X_{n-2-k}]&- \BE^2[X_k] -\BE^2[X_{n-2-k}] \\
& &+ \BE^2[X_k] +\BE^2[X_{n-2-k}]\\
& = \frac{2}{n-1}\sum_{k = 0}^{n-2} \var[X_k] + \frac{1}{n-1}\sum_{k = 0}^{n-2} \BE^2[X_k + X_{n-2-k}]\\
& = \frac{2}{n-1}\sum_{k = 0}^{n-2} \var[X_k] + \frac{1}{n-1}\sum_{k = 0}^{n-2} \BE^2 [X_n]\\
& = \frac{2}{n-1}\sum_{k = 0}^{n-2} \var[X_k] + \BE^2 [X_n]
\end{align*}
where we used the independence of $X_k$, $X_{n-2-k}$ and $Y_1$ in the second line, and the recurrence for expectation in the third line.

Plugging this result into the definition of the variance \ref{eq:var_def}, we obtain an equation identical to the one for the expectation (\ref{eq:E[X_n]}):
\begin{equation*}
\var[X_n] = \frac{2}{n-1}\sum_{k = 0}^{n-2} \var[X_k]
\end{equation*}

Thus, for the variance we also have 
\begin{equation}
\lim_{n\to \infty} \frac{\var X_n}{n} = e^{-2} 
\end{equation}

\subsection{Generating Function}
We conclude this part by a derivation of explicit expressions for the generating functions sums.  We are interested in an equation of the form
\begin{equation*}
X_n = X_k + X_{n-k-2}
\end{equation*}

where $X_k$ and $X_{n-k-2}$ are independent given $k$ as well as independent of $k$ itself. We introduce the generating function $f_n(\lambda) = \BE[e^{\lambda \cdot X_n}]$, and recall that the generating function of a sum of independent variables is a product of the generating functions.
\begin{equation*}
(n-1)f_n(\lambda) = \sum_{k=0}^{n-2} f_k(\lambda)\cdot f_{n-k-2}(\lambda)
\end{equation*}
to make it more transparent, let us define $m = n-2$ and rewrite the the above equation

\begin{equation} \label{eq:conv}
(m+1)f_{m+2}(\lambda) = \sum_{k=0}^{m} f_k(\lambda)\cdot f_{m-k}(\lambda)
\end{equation}
It is now clear that the right hand side is a convolution. This motivates the definition
\begin{equation*}
g(\lambda,t) = \sum_{m=0}^\infty f_m(\lambda)t^m
\end{equation*}
we use equation \ref{eq:conv} to obtain an ODE for $g$
\begin{equation*}
\frac{1}{t}[(m+2)f_{m+2}t^{m+1}] -\frac{1}{t^2}[f_{m+2}(\lambda)t^{m+2}] = t^{m}\sum_{k=0}^{m} f_k(\lambda)\cdot f_{m-k}(\lambda)
\end{equation*}
summing over $m$ gives:
\begin{align*}
\frac{1}{t}[\sum_{m = -1}^\infty (m+2)f_{m+2}(\lambda)\cdot t^{m+1} - f_1(\lambda)] -\frac{1}{t^2}[\sum_{m = -2}^{\infty} f_{m+2}(\lambda)t^{m+2} - f_1(\lambda)\cdot t -f_0(\lambda)] \\
 = \sum_{m = 0}^{\infty} t^{m}\sum_{k=0}^{m} f_k(\lambda)\cdot f_{m-k}(\lambda)
\end{align*}
Rearranging it we obtain
\begin{equation*}
\frac{1}{t}[\dd_t g(\lambda,t) - f_1(\lambda)] -\frac{1}{t^2}[g(\lambda,t) - f_1(\lambda)\cdot t -f_0(\lambda)] = g^2(\lambda,t)
\end{equation*}
Plugging $f_0(\lambda) = \BE[\exp(\lambda\cdot X_0)] = 1$ and multiplying by $t$ gives:

\begin{equation*}
\dd_t g(\lambda,t)- \frac{1}{t}[g(\lambda,t) -1] - t\cdot g^2(\lambda,t) = 0
\end{equation*}

This is a variant of Riccati's equation. Looking for a solution of the form:

\[ g(\lambda,t) = -\frac{y'}{y\cdot t }\]
where $y' = \dd_t y(\lambda, t)$, we obtain the following equation for $y$:
\[- \frac{y''(ty)-y'(y+ty')}{t^2y^2} - \frac{1}{t}[-\frac{y'}{ty}-1] - t\frac{(y')^2}{t^2y^2} = 0 \]
simplifying this we obtain

\[-y''+\frac{2}{t}y'+y = 0 \]
it can be easily verified that this equation has a solution

\[ y = C_1(1-t)e^t +C_2(1+t)e^{-t} \]
plugging this back into the definition of $y$ we obtain for $g$
\begin{equation*}
g(\lambda,t) = \frac{C_1e^t+C_2e^{-t}}{ C_1(1-t)e^t +C_2(1+t)e^{-t}}
\end{equation*}
Importantly $g$ depends only on the ratio $C_1/C_2$. Note that for $t = 0$ we obtain $g(\lambda,0)\equiv 1$ (this is simply the requirement that the equation exist). Differentiating $g$ we obtain 
\begin{equation*}
\dd_t g(\lambda,t) = \frac{C_1e^t-C_2e^{-t}}{ C_1(1-t)e^t +C_2(1+t)e^{-t}} +t\cdot\bigg(\frac{C_1(\lambda)e^t-C_2(\lambda)e^{-t}}{ C_1(1-t)e^t +C_2(1+t)e^{-t}}\bigg)^2
\end{equation*}
Setting $t = 0$ and recalling that $f_1 = e^\lambda $ we obtain
\begin{equation*}
\frac{C_1}{C_2} = \frac{1+e^{\lambda}}{1-e^{\lambda}} = \coth (\lambda/2)
\end{equation*} 
We choose $C_1 = \cosh (\lambda/2) $ and $C_2 = \sinh(\lambda /2)$
It is more convenient to work with $y$ instead of $g$. By definition 
\begin{equation*}
t\cdot g(t,\lambda) = \sum_{n=0}^\infty f_n(\lambda)t^{n+1} = -\dd_t (\ln y)
\end{equation*}
Therefore 
\[ f_n(\lambda) = -\frac{1}{(n+1)!}\dd_t^{(n+2)}( \ln y) \]

Fa\'a di Bruno's formula gives
\begin{align*}
\frac{d^n}{dx^n}\phi(h(x)) &= \sum \frac{n!}{m_1!m_2! \ldots m_n!}\cdot \phi^{(m_1+\ldots m_n)}\cdot \prod_{j=1}^n\bigg(\frac{h^{(j)}(x)}{j!} \bigg)^{m_j}\\
& = \sum_{k = 1}^n  \phi^{(k)}\cdot B_{n,k}(h', h'',\ldots, h^{(n-k+1)}) 
\end{align*}
where $B_{n,k}$ are the Bell polynomials given by 

\begin{align*}
B_{n,k}(x_1, x_2,\ldots, x_{n-k+1}) & = \sum \frac{n!}{j_1!j_2!\ldots j_{n-k+1}!}\left( \frac{x_1}{1!}\right)^{j_1}\left( \frac{x_2}{2!}\right)^{j_2}\left( \frac{x_{n-k+1}}{(n-k+1)!}\right)^{j_{n-k+1}}\\
\\
j_1+j_2+\ldots j_{n-k+1} & = k\\
j_1+2j_2+\ldots +(n-k+1)j_{n-k+1} & = n
\end{align*}

This explicit expression is unfortunately impractical for large $n$s. 


\section{General construction }
In the previous section we considered dimers on finite segment of sites, studying its limit as $n \to \infty$. We can try a different approach, namely, to construct the process directly on the line. The dimer-placement model definition is clearest in the abstract setting of a graph. We therefore present it in a setting slightly more general than intuitive. 

Let $G = (V,E)$ be a graph, for simplicity, assume it is of bounded degree $d$. We associate with each edge $e$ an independent random variable $\tau_e$ distributed uniformly on $[0,1]$. We think of $\tau_e$  as a ``wakeup time" for the edge. When the wakeup time occurs, we cover the edge and both the vertices it contains if none of them were previously covered (corresponding to adding the edge to $E(t)$. If either vertex is already covered we do nothing). 

First, let us illustrate this is indeed a generalization of the model we studied in the previous section. To see that note that the underlying graph is a segment of length $n$ and the edges are the natural edges which correspond to the gaps between sites. The fact that $\tau_e$ are i.i.d gives a uniform distribution on the order in which the edges are picked. 

Formally, the process we study defines a family of sets with the following property:

\begin{property} \label{constuction_on_graph} 
Given $\{ \tau_e\}_e$ define the family of sets $E(t)$ for $t\in [0,1]$ such that 

\begin{align*}
E(t) &= E(t_{-})\cup  \{ e \;| \text{ $\tau_e = t$ and $e\cap E(t_{-}) = \emptyset$ } \} 
\end{align*} 
where
\begin{align*}
E(t_{-}) = \cup_{s<t}E(s), \: E(0) = \emptyset
\end{align*} 
\end{property}

It is easy to see this a set family with this property, if it exists, generates a maximal independent set of edges at time 1. Namely every edge $e\in G$ is either in the set $E(t)$ for $t> \tau_e$ or it contains a vertex which belongs to an edge already in $E(t)$ at the time $t<\tau_e$. Indeed, if the edge $e$ does not belong to $E(t)$ it will be added to the family at time $\tau_e$, unless one of its vertices belong to an edge previously added to $E(t)$.

Our next goal is to establish the existence of $E(t)$. We will show that for almost all $\{\tau_e \}_e$ the family $E(t)$ exists. Uniqueness follows from the construction. 

Intuitively, the reason a set is ill-defined would be some elements for which it is not clear if they belong to the set or not. To address this let us consider for every edge $e$ the function

\begin{align*}
1_e(t) = 
\begin{cases}
1 \text{ if $e\in E(t)$ at time $t$ }\\
0 \text{ otherwise}
\end{cases}
\end{align*}
and the truncated functions

\begin{align*}
1^{(r)}_e(t) = 
\begin{cases}
1 \text{ if $e\in E_r(t)$ at time $t$ when considering only the edges in $B_r(e)$ }\\
0 \text{ otherwise}
\end{cases}
\end{align*}
where we define $B_r(e)$, the ball of radius $r$ around an edge $e = (u,v)$ to be 
\[ B_r(e) = B_r(u)\cup B_r(v) \]
($B_r(v)$ is the the sub graph of $G$ which is at graph distance at most $r$ from $v$).

Theorem \ref{thm:well_def} shows that a.s. the truncated functions converge to a limit which we identify with $1_e(t)$
\begin{align*}
1^{(r)}_e \stackrel{r\to\infty}{\longrightarrow} 1_e(t)
\end{align*}

The reason $1^{(r)}_e$ could fluctuate as $r$ grows is a ``cascade of tiles" that changes the state of an edge, making it impossible to determine from a finite ball around it whether the edge is present in the cover. If such fluctuations do not subside for arbitrarily large $r$s the event of an edge belonging to $E(t)$ unmeasurable in the graph topology and the process ill defined. 

We start with a definition that captures the intuition of a cascade of tiles:

\begin{definition}
Let $ \gamma = (e_1,e_2,\ldots, e_n)$ where $e_i$ edges of $G$ be a path in $G$. We call $\gamma$ \textit{monotone} if for all $i<j $ and $e_i,e_j\in \gamma$ holds $\tau_{e_i} > \tau_{e_j}$. 
\end{definition}

\begin{remark} Note that there is a natural partial order by inclusion for the set of monotone paths $\gamma_1\leq \gamma_2$ if $\gamma_1\subset \gamma_2$. Unfortunately, a monotone path needs not to be simple, however it is clear from the definition that any monotone $\gamma$ cannot contain the same edge twice.
\end{remark}
For an edge $e$ with time $\tau_e$ it is clear from the definition of our process, that only tiles with times preceding $\tau_e$ can influence the event of laying the tile corresponding to $e$. Therefore only edges lying on monotone path starting at $e$ can influence the event $e\in E(t)$ and thus the functions $1^{(r)}_e(t)$.

The following lemma gives an estimate on the length of a monotone path.

\begin{lemma}
Let $\gamma$ be a path of length $n$.

\[ \pr (\gamma \text{ is monotone}) = \frac{1}{n!} \]
\end{lemma} \label{lem:path}
\begin{proof}
Let $\gamma = (e_1, e_2, \ldots, e_n)$. Consider the event that $\gamma$ is monotone, this is given by the event
\[ A = \{\tau_{e_1}> \tau_{e_2}> \ldots > \tau_{e_n}\} \]

This is a decreasing sequence of i.i.d. random variables. The probability of this event is the same as that of any random permutation 
\[ \pr (\tau_{e_1}> \tau_{e_2}> \ldots > \tau_{e_n} ) = \pr(\tau_{\sigma(e_1)}> \tau_{\sigma (e_2)}> \ldots > \tau_{\sigma(e_n)}) \]
The number of permutations on $n$ elements is $n!$. 
\end{proof}

The previous discussion implies the following natural definition of the set of monotone paths having a first edge $e$,
\[ \Gamma_e = \{ \gamma|\; \gamma = (e_1,e_2,\ldots , e_n) \text{ is monotone and }  e_1 = e \} \]


We claim that $\Gamma_e$ contains only finite paths a.s. Actually, an even more general statement is true, all monotone paths in a bounded degree graph are finite.

\begin{lemma} \label{cor:finite_rad}
Almost surely, there exists $r<\infty$ such that $\Gamma_e \subset B_r(e)$.
 \end{lemma}
\begin{proof} 
\[ \pr( \Gamma_e \nsubseteq B_r(v)) \leq \pr( \text{exists a monotone path $\gamma$ starting with $e$ s.t.  }|\gamma|>r ) \]
Plugging in our assumption of the uniform bound $d$ on the vertex degree, we can bound the number of paths of length $r$ beginning with $e$ by $d^r$ and thus obtain a union bound of this event:
\[ \pr( \Gamma_e \nsubseteq B_r(e)) \leq \frac{d^r}{r!} \] 
Note that the probabilities of the events $\Gamma_e \nsubseteq B_r(v)$ have a finite sum
\[ \sum_r \pr(  \Gamma_e \nsubseteq B_r(e)) < \sum_r \frac{d^r}{r!}  = e^d <\infty \] 
therefore by the Borel Cantelli lemma only finitely many of them happen a.s. Therefore $\Gamma_e\subset B_r(e)$ for some finite random $r$ a.s.
\end{proof}

Armed with this result we can now prove 

\begin{theorem} \label{thm:well_def}
For almost all $\{ \tau_e\}_{e\in G}$ the family $E(t)$ defined by property \ref{constuction_on_graph} exists and is unique.
\end{theorem}
\begin{proof}
It is enough to show that for any edge $e\in E$ the functions $1^{(r)}_e(t)$ converge a.s. This will in particular settle the question of the existence of $E(t)$ by giving an explicit construction of it. 

By lemma \ref{cor:finite_rad} for each edge $e$ almost surely there exists a finite radius $r_0$ such that $\Gamma \subset B_{r_0}(e)$. Thus for any $r_1,r_2>r_0$ holds

\[ 1^{(r_1)}_e(t) = 1^{(r_2)}_e(t) \]
This proves the convergence of the sequence $\{ 1^{(r)}_e(t)\}_r$ establishing the uniqueness. 

\end{proof}
 \section{ Monomer concentration results }\label{sec:monomer}

This section is dedicated to establishing results regarding the limiting behaviour of monomers. We begin with generic results, which easily generalize to bounded degree graphs, and then continue with results and proofs which are special to $\Z^d$. Since the generalization to bounded degree graphs is fairly clear when possible, we see no reason to complicate the notation for the sake of seemingly added generality therefore we work throughout with the graph of $\Z^d$. For the sake of clarity, we separate the $\Z^d$ specific results from the more generic ones.

Consider a box of side-length $n$ which we denote by $\Lambda_n$.  Denote by 
\begin{align}
S_n = \frac{1}{|\Lambda_n|}\sum_{v\in\Lambda_n} (X_v -\BE X_v)
\end{align}
We evaluate the probability 
\[ \pr(|S_n| >\epsilon) \]
Let us define a truncated event $Y_v^{r}$ to be the event $v$ is not covered by a dimer when considering only dimers in a ball or radius $r$ around $v$. 
Using the truncated variables we can bound this probability:

\begin{align*}
\pr(|S_n| >\epsilon) & \leq   \pr(\left| \sum_{v\in\Lambda_n} Y^r_v - |\Lambda_n|\cdot\BE Y^r_v \right| > \frac{\epsilon}{3} \cdot |\Lambda_n|) + \\
& \pr (\sum_{\Lambda_n} |X_v -Y^r_v| > \frac{\epsilon}{3} \cdot |\Lambda_n|) + \pr(|\sum_{v\in\Lambda_n} |\BE Y^r_v -\BE X_v| > \frac{\epsilon}{3} \cdot |\Lambda_n|)
\end{align*}

From the proof of lemma \ref{cor:finite_rad} we have
\[ \pr(X_v\neq Y_v^{r}) \leq \frac{d^r}{r!} \]
Thus,
\begin{align*}
\BE (|\sum_{\Lambda_n} X_v - Y_v^{(r)}|) &\leq \BE (\sum_{\Lambda_n} |X_v - Y_v^{(r)}|)\\
& =  \sum_{\Lambda_n} \BE(|X_v - Y_v^{(r)}|) \leq |\Lambda_n|\cdot \frac{d^r}{r!}
\end{align*}
This bounds the last summand, in particular by

\begin{align*}
\pr (|\sum_{v\in\Lambda_n} |\BE Y^r_v -\BE X_v| > \epsilon \cdot |\Lambda_n|) \leq 1_{\frac{d^r}{r!}\leq \frac{\epsilon}{3} }
\end{align*}
By Markov's inequality we obtain
\begin{align*}
\pr (\sum_{\Lambda_n} |X_v -Y^r_v| > \frac{\epsilon}{3} \cdot |\Lambda_n|) \leq \frac{3}{\epsilon} \cdot \frac{d^r}{r!}
\end{align*}
which bounds the second summand. Note that these bounds are uniform in $n$ and tend to 0 as $r\to\infty$. 

For the first summand we note that 
\begin{align*}
\sum_{\Lambda_n} Y^r_v - |\Lambda_n|\cdot\BE Y^r_v = \sum_{i\in (0,r)^d}\sum_{\substack{v\in\Lambda_n \\ v\equiv i\mod(r)}} Y^r_v - |\Lambda_n|\cdot\BE Y^r_v
\end{align*}
The inner sum is a sum of i.i.d. Bernoulli random variables therefore by the Chernoff bound 
\begin{align*}
\pr(\sum_{\substack{v\in\Lambda_n \\ v\equiv i\mod(r)}} Y^r_v - |\Lambda_{\frac{n}{r}}|\cdot\BE Y^r_v > \frac{\epsilon}{3}\cdot |\Lambda_{\frac{n}{r}}|) \leq e^{- (\epsilon /3)^2 \cdot |\Lambda_{\frac{n}{r}}|}
\end{align*}
taking a union bound we obtain 
\begin{align*}
\pr (\sum_{\Lambda_n} |X_v -Y^r_v| >\frac{\epsilon}{3}\cdot |\Lambda_n|)\leq r^d \cdot e^{-(\epsilon/3)^2\cdot |\Lambda_{\frac{n}{r}}|}
\end{align*}
summing the estimates we obtain
\begin{align*}
\pr (|S_n|>\epsilon)\leq r^d \cdot e^{-(\epsilon/3)^2\cdot |\Lambda_{\frac{n}{r}}|} + \frac{3}{\epsilon}\cdot \frac{d^r}{r!} + 1_{\frac{d^r}{r!}\leq \frac{\epsilon}{3}} 
\end{align*}

Taking $r=\log n$ for $n$ large enough (so as $ 1_{\frac{d^r}{r!}\leq \frac{\epsilon}{3}} = 0$) we use Stirling's formula to obtain

\begin{align*}
\pr (|S_n|>\epsilon)\leq \log^d n \cdot e^{-\epsilon^2\cdot n^{d-1}} + (\frac{d}{\log n})^{\log n}  
\end{align*}

\subsection{a CLT result for $\Z^d$}
First let us show how for the $\Z^d$ case the stationarity of our process implies a positive concentration of monomers. This is a direct consequence of Templeman's ergodic theorem for $\Z^d$ actions 
(see \cite[chapter 2, theorem 2.6]{sarig2009lecture}) which we cite here for completeness: 
\begin{theorem}[Tempelman 1975]\label{thm:Tempelman}
Let $T_1,\ldots,T_d$ be d-commuting probability preserving maps on a probability space, and suppose $\{I_r\}_{r\geq 1}$ is an increasing sequence of boxes which tends to $\Z^d_+$. If $f\in L^1$, then 
\begin{align}
\frac{1}{|I_r|} \sum_{v \in I_r} f\circ T^{v} \stackrel{r\to\infty}{\longrightarrow}\BE(f|Inv(T_1)\cap\ldots\cap Inv(T_d)) \; a.s.
\end{align}

where we sum over $v$ as vectors in $\Z^d$ namely $T^v = T_1^{v_1}\circ\ldots\circ T_d^{v_d}$ where $v_i$ is the i-th coordinate of $v$. $Inv(T)$ is the space of invariant functions under $T$ and $\BE(f|Inv(...))$ is the conditional expectation.
\end{theorem}

Using this we obtain
\begin{align*}
S_n \stackrel{n\to\infty}{\longrightarrow} 0 
\end{align*}

or in other words, 
\begin{align*}
\frac{1}{|\Lambda_n|} \sum_{v\in \Lambda_n} X_v \stackrel{n\to\infty}{\longrightarrow} \BE X_0 = \pr(X_0 = 1)
\end{align*}

Our goal is to establish a CLT result for $S_n$ namely we would like to show that 

\begin{align*}
\pr (\sqrt{|\Lambda_n|}S_n) \sim N(0, \sigma). 
\end{align*}
Consider $X_u$ and $X_v$, and denote the distance beween them by $r = ||u-v||$. For these variables there is a mixing property by lemma \ref{cor:finite_rad}, namely 

\begin{align*}
|\pr (X_u\cap X_v) - \pr(X_u)\pr(X_v)| \leq  \frac{d^r}{r!}
\end{align*}
(this follows from noticing that if $u\notin \max \Gamma_v $ makes $X_u,X_v$ independent, and plugging the estimate on the radius of $\max \Gamma_v$.)

A central limit theorem for our process follows from the theorem of Bolthausen \cite{bolthausen1982central} which we cite for completeness:
Let $X_\rho$ be a real valued stationary random field with $\rho\in \Z^d$, i.e. the $X_\rho$ are real random variables and the joint laws are shift invariant.

If $\Lambda\subset\Z^d$ let $\CA_{\Lambda}$ be the $\sigma$-algebra generated by $X_\rho$, $\rho\in\Lambda$. If $\Lambda_1,\Lambda_2\subset \Z^d$, let $d(\Lambda_1,\Lambda_2) = \inf\{ d(\rho_1,\rho_2): \rho_1\in \Lambda_1,\rho_2\in\Lambda_2 \}$. The mixing coefficients we use are defined as follows, if $n\in\BN$,$k,l\in\BN\cup\{ \infty\}$
\begin{align*}
\alpha_{k,l}(n) &= \sup \{|\pr (A_1\cap A_2) - \pr(A_1)\pr(A_2)| : A_i\in \CA_{\Lambda_i},\: |\Lambda_1|\leq k, |\Lambda_2|\leq l, d(\Lambda_1,\Lambda_2)\geq n \}\\
\rho(n) &= \sup \{ |\Cov(Y_1,Y_2)| : Y_i\in L_2(\CA_{\{ \rho_i \} }),\: \Vert Y_i\Vert_2\leq 1,\: d(\rho_1,\rho_2)\geq n \}
\end{align*}
\begin{theorem}
If $\sum_{m=1}^\infty m^{d-1}\alpha_{k,l}(m) <\infty$ for $k+l \leq 4$, $\alpha_{1,\infty}(m) = o(m^{-d})$ and if 
\begin{align*}
\sum_{m=1}^\infty m^{d-1}\rho(m) < \infty
\end{align*}
or 
\begin{align*}
\text{for some } \delta >0 \; \Vert X_{\rho}\Vert_{2+\delta}<\infty \text{ and } \sum_{m=1}^\infty m^{d-1}\alpha_{1,1}(m)^{\delta/(2+\delta)} < \infty
\end{align*}
Then $\sum_{\rho\in\Z^d} | \Cov (X_0, X_{\rho})|<\infty $ and if $\sigma^2 = \sum_{\rho} \Cov(X_0,X_\rho)>0$, then the laws of $ |\Lambda_n|^{1/2} \cdot S_n/\sigma $ converge to the standard normal one.
\end{theorem}

The finiteness of the sums involving the mixing coefficients is an immediate consequence of lemma \ref{cor:finite_rad}. It asserts that the sums contains only a finite number of summands.

The only non trivial matter is checking $\sigma^2 = \sum_{\rho} \Cov(X_0,X_\rho)>0$. The idea is seen most clearly for $\Z^2$, therefore we present it here for the plane grid, however it generalizes for arbitrary dimension $d$.

We call a ``cage" the event when for a rectangle frame enclosing a $4 \times 1$ segment is filled with tiles before the inside is being filled. See figure \ref{fig:cage} for illustration.

\begin{figure}
\includegraphics[height=8cm]{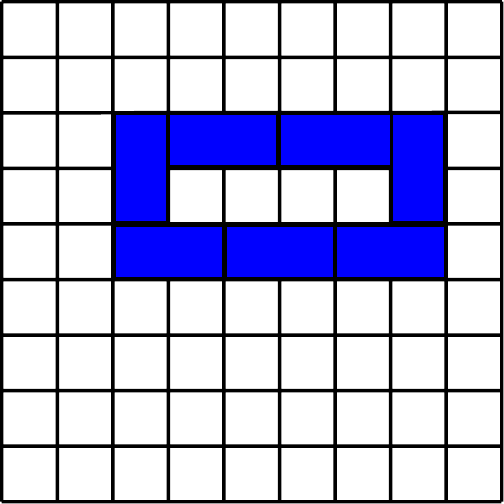}
\caption{a cage \label{fig:cage}}
\end{figure}

Our next observation is that the sites within the cage are covered independently of the sites outside the cage. Using Templeman's theorem \ref{thm:Tempelman} we can see there is a positive density of cages. Using this we can bound $\sigma$ from 0:
\begin{align*}
\var[|\Lambda_n|\cdot S_n] & \geq \BE [ \var[\sum_{v\in \Lambda_n} x_v|\tau] ] \\
& = \BE [ \var [ \sum_{v \text{ in cage}} x_v+ \sum_{v \text{ out of cage}} x_v | \tau ] ]\\
& \geq \BE [ \var [ \sum_{v \text{ in cage}} x_v | \tau ] ]\\
& \geq c \cdot \BE [\text{ \# of cages } ]
\end{align*}
where $\tau$ is the the $\sigma$-field created by conditioning the positions of cages in a box of size $n$. Note that by Tempelman's theorem, we know there is a positive density of cages. This bounds the sum of covariances from 0, allowing us to apply Bolthousen's theorem.

\section{Questions and discussion}

There is  another setting which should be solvable using a recurrence equation approach, this is a $d$-regular tree, which can be obtained as a limit of trees chopped at depth $n$. The recurrence in this case is more complicated than the recurrence we obtained for $\Z$, but the independence of the different parts of the graph holds in this case as well, which was the key element in the approach. 

This gives rise to a couple of questions:
The CLT result for our model followed from Bolthousen's theorem which uses the $\Z^d$ structure, and moreover in bounding the variance from 0 we used Templeman's theorem which heavily uses the $\Z^d$ structure. However, the essential ingredient for a CLT type of theorem is mixing or near independence, which we obtain from lemma \ref{cor:finite_rad} in our model. 

\begin{question}
 It would be interesting to see in what generality the CLT result is valid? In particular for Cayley graphs of what groups can such a result be obtained? 
\end{question}

One cannot expect to be able to establish these results for a general Cayley graph using the tools of ergodic theory. To apply ergodic theorems ameanability is obviously needed, however it is not clear such a result cannot be established by other means. 

The tail estimates we show in section \ref{sec:monomer} lead us to conjecture that the result should be valid at least for uniformly bounded degree graphs.

\author{Jacob J. Kagan} \\

\email{jacov.kagan@gmail.com }

\end{document}